\documentclass[preprint]{elsarticle}

\makeatletter
\def\ps@pprintTitle{%
	\let\@oddhead\@empty
	\let\@evenhead\@empty
	\def\@oddfoot{}%
	\let\@evenfoot\@oddfoot}
\makeatother

\usepackage{lineno,hyperref}

\usepackage{amsmath, amsthm, amsfonts}
\usepackage[only,llbracket, rrbracket,llparenthesis,rrparenthesis]{stmaryrd}

\theoremstyle{plain}
\newtheorem{thm}{Theorem}[section] 
\newtheorem{lem}[thm]{Lemma} 
\newtheorem{prop}[thm]{Proposition}

\theoremstyle{definition}
\newtheorem{defn}[thm]{Definition}
\theoremstyle{remark}
\newtheorem*{rem}{Remark}
\modulolinenumbers[5]

\journal{Topology and its Applications}

\begin{document}

\begin{frontmatter}

\title{Topology and experimental distinguishability}


\author[]{Christine A. Aidala}
\ead{caidala@umich.edu}

\author[]{Gabriele Carcassi\corref{cor1}}
\ead{carcassi@umich.edu}

\author[]{Mark J. Greenfield}
\ead{markjg@umich.edu}

\address{University of Michigan, Ann Arbor, MI 48109, USA}

\cortext[cor1]{Corresponding author}

\begin{abstract}
	In this work we introduce the idea that the primary application of topology in experimental sciences is to keep track of what can be distinguished through experimentation. This link provides understanding and justification as to why topological spaces and continuous functions are pervasive tools in science. We first define an experimental observation as a statement that can be verified using an experimental procedure and show that observations are closed under finite conjunction and countable disjunction. We then consider observations that identify elements within a set and show how they induce a Hausdorff and second-countable topology on that set, thus identifying an open set as one that can be associated with an experimental observation. We then show that experimental relationships are continuous functions, as they must preserve experimental distinguishability, and that they are themselves experimentally distinguishable by defining a Hausdorff and second-countable topology for this collection. 
\end{abstract}


\end{frontmatter}


\section{Introduction}

The successful use of mathematical ideas in experimental sciences is long established and celebrated~\cite{wigner}. Topology is perhaps the most widespread, either directly~\cite{rashevsky, chichilnisky} or as a foundation to other tools~\cite{nakahara, basener}.  This leads one to ask, why is it so successful? What property is captured by topological spaces that is so fundamental for scientific investigation?

We believe experimental distinguishability to be the relevant concept. The notion of ``nearness" captured by topologies keeps track of how hard it is to tell two elements apart via scientific observation. The standard topology on the real line, for example, captures the inability to measure a continuous value with infinite precision. This connection holds in cases where a metric would not be physically meaningful (e.g.~phase space) and in discrete spaces (where each element can be individually identified).  It is fitting, then, that the use of topology is so widespread, as defining what can be experimentally distinguished is a fundamental aspect of science.

The aim of this paper is to lay down a framework that formalizes this insight. We first define experimental observations as statements paired with an experimental test that is able to verify them. We then study their properties under logical operations and conclude that they are only closed under finite conjunction and countable disjunction. Next we define experimental distinguishability and show, as our first main result, that any set of objects that are experimentally distinguishable is a Hausdorff, second-countable topological space, where the topology is formally defined in terms of the observations themselves. Last we define experimental relationships between experimentally distinguishable objects and show, as our second main result, that experimental relationships are represented by continuous functions and are themselves experimentally distinguishable.

\section{Experimental observations}

In science, a statement can be accepted as true only if there exists a way to experimentally verify it. To capture this notion we introduce the following definitions.

\begin{defn}
	A \textbf{statement} $\mathsf{s}$ is a declarative sentence that is either true or false, as in classical logic. 
\end{defn}

\begin{defn}
	An \textbf{experimental test} $\mathsf{e}$ is a repeatable procedure (i.e. it can be restarted and stopped an arbitrary number of times) which may be successful, in which case it terminates in finite time, or may be not successful, in which case it may or may not terminate.\footnote{In line with philosophical tradition\cite{mill,russell}, we can define this experimental test:
		\begin{enumerate}
			\item Find a swan
			\item If black terminate successfully
			\item Go to step 1
	\end{enumerate}}
\end{defn}

\begin{defn}
	An \textbf{experimental observation} $\mathsf{o}$ is a tuple $\llparenthesis \mathsf{s}, \mathsf{e} \rrparenthesis$ consisting of a statement $\mathsf{s}$ and an experimental test $\mathsf{e}$ such that the statement is true if and only if the  experimental test is successful. The experimental observation is \textbf{verified} if the statement is true.
\end{defn}

\section{Algebra of experimental observations}

We now want to understand how the standard Boolean algebra defined on statements carries over to experimental observations.

\begin{rem}
	Experimental observations are not closed under negation. The existence of an experimental test to verify a statement does not imply the existence of an experimental test to verify its negation.
\end{rem}

\begin{defn}
	The \textbf{conjunction} or \textbf{logical AND} of a finite collection  of experimental observations $\{\mathsf{o}_i\}_{i=1}^{n}=\{\llparenthesis \mathsf{s}_i, \mathsf{e}_i\rrparenthesis\}_{i=1}^{n}$ is the experimental observation $\bigwedge\limits_{i=1}^{n} \mathsf{o}_i = \llparenthesis \mathsf{s}, \mathsf{e}\rrparenthesis$ where $\mathsf{s} = \bigwedge\limits_{i=1}^{n} \mathsf{s}_i$ is the conjunction of the respective statements and $\mathsf{e} = \mathsf{e}_\wedge(\{\mathsf{e}_i\}_{i=1}^{n})$ is the experimental test that successfully terminates if and only if all $\{\mathsf{e}_i\}_{i=1}^{n}$ successfully terminate.
\end{defn}

\begin{proof}
	To show that the conjunction is well defined it suffices to show that we can construct a suitable experimental test.  Let $\mathsf{e}_\wedge=\mathsf{e}_\wedge(\{\mathsf{e}_i\}_{i=1}^{n})$ be the experimental procedure defined as follows:
	\begin{enumerate}
	\item for each $i=1,\ldots,n$ run the test $\mathsf{e}_i$
	\item if all tests $\mathsf{e}_i$ terminate successfully then terminate successfully
	\end{enumerate}
	This experimental procedure terminates successfully if and only if all $\mathsf{e}_i$ terminate successfully. It will do so in finite time as each of the finitely many $\mathsf{e}_i$ succeeds in finite time. Therefore $\mathsf{e}_\wedge(\{\mathsf{e}_i\}_{i=1}^{n})$ is an experimental test that is successful if and only if all statements $\{\mathsf{s}_i\}_{i=1}^{n}$ are true. So, $\bigwedge\limits_{i=1}^{n} \mathsf{o}_i = \llparenthesis\bigwedge\limits_{i=1}^{n} \mathsf{s}_i, \mathsf{e}_{\wedge}(\mathsf{e}_i)\rrparenthesis$ is an experimental observation.
\end{proof}

\begin{rem}
	Conjunction cannot be extended to a countable collection as verification would require infinite time.
\end{rem}

\begin{defn}
	The \textbf{disjunction} or \textbf{logical OR} of a countable (finite or infinite) collection of experimental observations $\{\mathsf{o}_i\}_{i=1}^{\infty}=\{\llparenthesis \mathsf{s}_i, \mathsf{e}_i\rrparenthesis\}_{i=1}^{\infty}$ is the experimental observation $\bigvee\limits_{i=1}^{\infty} \mathsf{o}_i = \llparenthesis \mathsf{s}, \mathsf{e}\rrparenthesis$ where $\mathsf{s} = \bigvee\limits_{i=1}^{\infty} \mathsf{s}_i$ is the disjunction of the respective statements and $\mathsf{e} = \mathsf{e}_\vee(\{\mathsf{e}_i\}_{i=1}^{\infty})$ is the experimental test that successfully terminates if and only if at least one experimental test in $\{\mathsf{e}_i\}_{i=1}^{\infty}$ successfully terminates.
\end{defn}

\begin{proof}
	As before, it suffices to show that we can construct a suitable experimental test. Let $\mathsf{e}_\vee=\mathsf{e}_\vee(\{\mathsf{e}_i\}_{i=1}^{\infty})$ be the experimental procedure defined as follows:
	\begin{enumerate}
	\item initialize $n$ to 1
	\item for each $i=1,\ldots,n$:
	\begin{enumerate}
		\item run the test $\mathsf{e}_i$ for $n$ seconds
		\item if $\mathsf{e}_i$ terminated successfully then terminate successfully
	\end{enumerate}
	\item increment $n$ and go to step 2
	\end{enumerate}
	Suppose there exists an $i \in \mathbb{Z}^+$ such that $\mathsf{e}_i$ will terminate successfully. Then the above procedure will eventually run that test for sufficient time for it to terminate successfully. It will do so in finite time as it will have run finitely many tests finitely many times each for a finite amount of time. Therefore $\mathsf{e}_\vee(\{\mathsf{e}_i\}_{i=1}^{\infty})$ is an experimental test that is successful if and only if at least one statement in $\{\mathsf{s}_i\}_{i=1}^{\infty}$ is successful. So $\bigvee\limits_{i=1}^{\infty} \mathsf{o}_i =\llparenthesis\bigvee\limits_{i=1}^{\infty} \mathsf{s}_i, \mathsf{e}_{\vee}(\{\mathsf{e}_i\}_{i=1}^{\infty})\rrparenthesis$ is an experimental observation.
\end{proof}

Taken together, finite conjunction and countable disjunction form the algebra of experimental observations. We also introduce the following special case, which will be useful later.

\begin{defn}
Any experimental observation whose statement is a contradiction is also called a \textbf{contradiction} and is noted by $\bot$.
\end{defn}

\begin{defn}
Two experimental observations $\mathsf{o}_1$ and $\mathsf{o}_2$ are said to be \textbf{incompatible} if the conjunction $\mathsf{o}_1\wedge\mathsf{o}_2$ is a contradiction.
\end{defn}

\section{Experimental domain}

We now want to characterize the sets of observations for which it is feasible to experimentally verify all true statements.

\begin{defn}
	An \textbf{experimental domain} is a set of observations closed under finite conjunction and countable disjunction, such that all observations can be tested in infinite time. 
\end{defn}

\begin{rem}
	We do allow infinite time for the verification of a domain with the understanding that some domains will only be partially verified in finite time. As we have, so to speak, only one infinity to spend, we spend it here to maximize its utility.
\end{rem}

At this point, the similarities between this mathematical structure and topologies are starting to emerge. In analogy to the latter, we define the following.

\begin{defn}
	A \textbf{sub-basis} of an experimental domain is any subset that can generate all others via finite conjunction and countable disjunction. A \textbf{basis} of an experimental domain is any subset that can generate all others by countable disjunction.
\end{defn}

\begin{rem}
	As for topologies, given a sub-basis one can generate a basis by taking all finite conjunctions. Any infinite sub-basis will generate a basis of the same cardinality.
\end{rem}

\begin{prop}
Let $\mathcal{D}$ be an experimental domain. Then there exists a countable basis (equivalently, sub-basis) $\mathcal{B}$ of $\mathcal{D}$.
\end{prop}

\begin{proof}
If there exists a countable basis $\mathcal{B}$, then given infinite time one can test all observations in $\mathcal{B}$. Given which observations of the basis are verified, one can deduce which other observations in $\mathcal{D}$ are verified (again using infinite time) by computing the appropriate disjunctions. 

If there does not exist a countable basis, then by definition there does not exist a sequence of experimental observations in $\mathcal{D}$ from which one can deduce all other observations in $\mathcal{D}$. Hence it is impossible to test all members of $\mathcal{D}$.
\end{proof}

\section{Experimental distinguishability}

We now turn our attention to a more specific case. We want to characterize an experimental domain whose purpose is to identify an element among a set of possibilities.

\begin{defn}
	An \textbf{experimental identification domain} is the triplet \linebreak $(\mathcal{D}_X, X, x)$ where:
	\begin{itemize}
		\item $X$ is the set of \textbf{possibilities}, and satisfies $|X|>1$
		\item $x$ is the element to identify among the possibilities, therefore $x \in X$
		\item $\mathcal{D}_X$ is an experimental domain containing all possible experimental observations of the form $\mathsf{o} = \llparenthesis x\in U, \mathsf{e}_\in(U)\rrparenthesis$ where $U \subseteq X$ is a set of possibilities and $\mathsf{e}_\in(U)$ is an experimental test that succeeds if and only if $x \in U$
	\end{itemize}
	Any subset $U\subseteq X$ for which such an observation exists is said to be a \textbf{verifiable set}.
\end{defn}

\begin{lem}
\label{setbehavior}
	Let $U_1, U_2, ... , U_n, ...$ be a countably infinite sequence of verifiable sets. The finite intersection $\bigcap\limits_{i=1}^{n} U_i$ and the countable union $\bigcup\limits_{i=1}^{\infty} U_i$ are verifiable sets.
\end{lem}

\begin{proof}
	We show that the finite intersection of verifiable sets is a verifiable set. Let $U_1, U_2, ... , U_n \subseteq X$ be $n$ verifiable sets. For each $U_i$ there exists an experimental observation $\mathsf{o}_i = \llparenthesis x\in U_i, \mathsf{e}_\in(U_i) \rrparenthesis$. Consider the finite conjunction $\mathsf{o} = \bigwedge\limits_{i=1}^{n} \mathsf{o}_i = \llparenthesis \bigwedge\limits_{i=1}^{n} x\in U_i , \mathsf{e}_{\wedge}(\mathsf{e}_\in(U_i)) \rrparenthesis=\llparenthesis x\in \bigcap\limits_{i=1}^{n} U_i, \mathsf{e}_\in(\bigcap\limits_{i=1}^{n} U_i)\rrparenthesis$ which is an experimental observation. Thus $\bigcap\limits_{i=1}^{n} U_i$ is a verifiable set.
	
	We show that the countable union of verifiable sets is a verifiable set. Let $U_1, U_2, ... , U_n, ... \subseteq X$ be an infinite sequence of verifiable sets. For each $U_i$ there exists an experimental observation $\mathsf{o}_i = \llparenthesis x\in U_i, \mathsf{e}_\in(U_i)\rrparenthesis$. Consider the infinite disjunction $\mathsf{o} = \bigvee\limits_{i=1}^{\infty} \mathsf{o}_i = \llparenthesis\bigvee\limits_{i=1}^{\infty} x\in U_i, \mathsf{e}_{\vee}(\mathsf{e}_\in(U_i))\rrparenthesis=\llparenthesis x\in \bigcup\limits_{i=1}^{\infty} U_i, \mathsf{e}_\in(\bigcup\limits_{i=1}^{\infty} U_i)\rrparenthesis$ which is an experimental observation. Thus $\bigcup\limits_{i=1}^{\infty} U_i$ is a verifiable set.
\end{proof}

To make sure we have enough experimental observations to tell the possibilities apart, we introduce the following definition.

\begin{defn}
A set of possibilities $X$ is \textbf{experimentally distinguishable} if there exists an experimental identification domain $(\mathcal{D}_X, X, x)$ such that for any two possibilities $x_1, x_2 \in X$ we can find two incompatible experimental observations $\llparenthesis x\in U_1, \mathsf{e}_\in(U_1)\rrparenthesis, \llparenthesis x\in U_2, \mathsf{e}_\in(U_2)\rrparenthesis\in\mathcal{D}_X$ such that $x_i\in U_i$ for $i=1,2$. 
\end{defn}

We are now ready to prove the first main result of this work.

\begin{thm}
A set of experimentally distinguishable possibilities $X$ has a natural Hausdorff, second-countable topology $(X,\mathsf{T})$ with the open sets given by the verifiable sets of the associated experimental identification domain $(\mathcal{D}_X, X, x)$.
\end{thm}
\begin{proof}
First, from the definition one can see that for all $x\in X$, there exists a verifiable set $U$ with $x\in U$. Therefore the union of a (countable) basis is the verifiable set $X$. Further, there exist at least two incompatible experimental observations, corresponding to two disjoint sets, so the empty set is a verifiable set. Now, \ref{setbehavior} shows that the collection $\mathsf{T}$ is closed under finite intersection and countable union. Because $\mathsf{T}$ is determined by an experimental domain, there is a countable basis of observations which translates to a countable basis of open (verifiable) sets, so it is second-countable. To show it is closed under arbitrary union, notice that an arbitrary union may be rewritten as a union of basis elements, which is then a countable union, and so it remains in $\mathsf{T}$. That the topology is Hausdorff is immediate from the last part of the definition. 
\end{proof}

\begin{rem}
As any Hausdorff, second-countable topological space has at most cardinality of the continuum, that is also the greatest cardinality that a set of experimentally distinguishable objects can have. We can conclude that sets of mathematical objects, such as all functions from $\mathbb{R}$ to $\mathbb{R}$, that do not satisfy this requirement are not good candidates to represent scientific concepts.
\end{rem}

\section{Experimental relationships}

We now want to characterize relationships between two experimentally distinguishable elements. Such relationships can be defined either on the values (i.e.~the possibilities) or on the observations (i.e.~the experimental domain). We need to show that both definitions lead to the same mathematical object.

\begin{defn}[Experimental relationship between possibilities]
	Let \linebreak $(\mathcal{D}_X, X, x)$ and $(\mathcal{D}_Y, Y, y)$ be two experimental identification domains. An \textbf{experimental relationship} is a map $f : X \rightarrow Y$ that can be used within an experimental test.
\end{defn}

\begin{prop}
	The experimental relationship $f$ defined above is a continuous function.
\end{prop}
\begin{proof}
Let $o_y = \llparenthesis y\in U_Y,\mathsf{e}_{\in}(U_Y)\rrparenthesis \in \mathcal{D}_Y$.  Since $f$ can be used within an experimental test, consider the following experimental procedure:
\begin{enumerate}
	\item map $x$ to $y=f(x)$
	\item run the test $\mathsf{e}_{\in}(U_Y)$
\end{enumerate}
It will be successful if and only if $y \in U_Y$. Since $y=f(x)$, it is successful if and only if $x \in f^{-1}(U_Y)$: the procedure is the test $\mathsf{e}_{\in}(f^{-1}(U_Y))$ on $x$. This means that $o_x = \llparenthesis x \in f^{-1}(U_Y),\mathsf{e}_{\in}(f^{-1}(U_Y))\rrparenthesis$ is an experimental observation and it must be in $\mathcal{D}_X$ since $\mathcal{D}_X$ contains all possible experimental observations of that form. It follows that $f^{-1}(U_Y)$ is a verifiable set and must be part of the topology $\mathsf{T}_X$.
\end{proof}

\begin{defn}[Experimental relationship between observations]
	Let $(\mathcal{D}_X, X, x)$ and $(\mathcal{D}_Y, Y, y)$ be two experimental identification domains. An \textbf{experimental relationship} is a map $g : \mathcal{D}_Y \rightarrow \mathcal{D}_X$ such that if $\mathsf{o} \in \mathcal{D}_Y$ is verified then $g(\mathsf{o}) \in \mathcal{D}_X$ must also be verified. To be consistent, such a relationship must have these properties:
	\begin{enumerate}
	\item it is compatible with conjunction and disjunction: for any $\mathsf{o}_1, \mathsf{o}_2 \in \mathcal{D}_y$, we have $g(\mathsf{o}_1 \wedge \mathsf{o}_2)=g(\mathsf{o}_1)\wedge g(\mathsf{o}_2)$ and $g(\mathsf{o}_1 \vee \mathsf{o}_2)=g(\mathsf{o}_1)\vee g(\mathsf{o}_2)$
	\item contradiction leads to contradiction: $g(\bot) = \bot$
	\item no knowledge leads to no knowledge: if $Y$ is the verifiable set associated with $\mathsf{o} \in \mathcal{D}_Y$ then $X$ is the verifiable set associated with  $g(\mathsf{o})$
	\end{enumerate}
\end{defn}

\begin{prop}
	For each experimental relationship $g$ defined above there exists a unique continuous function $f: X \rightarrow Y$ such that $g(\llparenthesis y\in U_Y,\mathsf{e}_{\in}(U_Y)\rrparenthesis) = \llparenthesis x\in f^{-1}(U_Y),\mathsf{e}_{\in}(f^{-1}(U_Y))\rrparenthesis$.
\end{prop}

\begin{proof}
	First we reformulate $g$ in terms of the open set. We can redefine $g: \mathsf{T}_Y \rightarrow \mathsf{T}_X$ to be the map between the verifiable sets corresponding to the experimental observations. This map has the following properties:
	\begin{enumerate}
		\item it is compatible with union and intersection, i.e. for any subsets $V_1, V_2 \subseteq Y$, we have $g(V_1 \cap V_2)=g(V_1)\cap g(V_2)$ and $g(V_1 \cup V_2)=g(V_1)\cup g(V_2)$
		\item $g(\emptyset) = \emptyset$
		\item $g(Y) = X$
	\end{enumerate}

	Then we construct the unique extension $\bar{g}:\sigma_Y\to\sigma_X$ to the Borel $\sigma$-algebras of $X$ and $Y$, respectively $\sigma_X$ and $\sigma_Y$, such that $\bar{g}|_{\mathsf{T}_Y}=g$ and $\bar{g}$ is compatible with union, intersection and complements. Let $\bar{g}(V) = g(V)$ for all open sets $V \in \mathsf{T}_Y$. Let $A \in \sigma_Y$ (not necessarily open) and $A^C$ be its complement. We must have $\bar{g}(A^C) = \bar{g}(A)^C = X\setminus \bar{g}(A)$ for $\bar{g}$ to be compatible with complements. Recall that all Borel sets in $\sigma_Y$ and $\sigma_X$ may be written as some combination of unions, intersections, and complements of open sets. Thus, the construction uniquely determines what $\bar{g}$ should output on any Borel set. We need only check that the output is still a Borel set. But by definition of $\bar{g}$, the outputs will be given as unions, intersections, and complements of outputs of $g$, which are open sets, and so the image of $\bar{g}$ is contained in $\sigma_X$.  $\bar{g}$ is well defined.
	
	Next we define $\hat{g}:Y\to\sigma_X$ such that $\hat{g}(y) = \bar{g}(\{y\})$. Since $Y$ is Hausdorff, every singleton $\{y\}$ is closed and is therefore a Borel set. Therefore $\bar{g}(\{y\})$ is well defined and so is $\hat{g}(y)$.
	
	We claim that $\hat{g}(y_1)\cap\hat{g}(y_2) = \emptyset$ if and only if $y_1\neq y_2$ for all $y_1,y_2\in Y$ such that $\hat{g}(y_i)\neq\emptyset$ for $i=1,2$. If $y_1\neq y_2$ we have
	$$
	\hat{g}(y_1)\cap\hat{g}(y_2) = \bar{g}(\{y_1\})\cap\bar{g}(\{y_2\}) = \bar{g}(\{y_1\}\cap\{y_2\}) = \bar{g}(\emptyset) = \emptyset.
	$$
	Conversely, if $y_1 = y_2$ we have
	$$
	\hat{g}(y_1)\cap\hat{g}(y_2) = 	\hat{g}(y_1)\cap\hat{g}(y_1) = 
	\hat{g}(y_1) \neq \emptyset.
	$$
	
	We are now ready to define $f: X\to Y$ such that $f(x) = y$ if and only if $x\in \hat{g}(y)$. Since $g(Y)=X$, there exists $y\in Y$ such that $x\in\hat{g}(y)$. By the preceding claim, this $y$ is unique. $f: X\to Y$ is well defined. Note that no arbitrary choices were made that led to the construction of $f$, which is therefore determined uniquely by $g$. 
	
	Now we show that $g = f^{-1} |_{\mathsf{T}_Y}$. Let $V\in\mathsf{T}_Y$. We want to show $f^{-1}(V) = g(V)$. Let $x\in f^{-1}(V)$. Then for some $y \in V$ we have $f(x)=y$. $x\in \hat{g}(y)$ by construction of $f$. $\hat{g}(y) \subset g(V)$ since $\{y\}\subset V$, so $x\in g(V)$. $f^{-1}(V) \subseteq g(V)$. Conversely, let $x\in g(V)=\bar{g}(V)$. Then for some $y\in V$, we have $x\in\bar{g}(\{y\})\subset\bar{g}(V)$. But then by definition we have $f(x)=y$, so $x\in f^{-1}(V)$. $f^{-1}(V) \supseteq g(V)$. $f^{-1}(V) = g(V)$ for all $V\in\mathsf{T}_Y$ and therefore $f^{-1}|_{\mathsf{T}_Y}=g$.
	
	Lastly we claim $f$ is continuous. It is so since $g = f^{-1} |_{\mathsf{T}_Y}$ takes open sets to open sets. 
\end{proof}

We can now state the second main result of this work.

\begin{thm}
	An experimental relationship between two sets $X$ and $Y$ of experimentally distinguishable possibilities is a continuous function $f : X \rightarrow Y$ between the respective natural topologies $(X,\mathsf{T}_X)$ and $(Y,\mathsf{T}_Y)$. 
\end{thm}

\begin{proof}
	As we saw in the previous results, both definitions lead to experimental relationships being fully characterized by a continuous function.
\end{proof}

\begin{rem}
	This result gives a formal justification as to why continuous functions are prevalent in science in general and in physics in particular. As topologies capture experimental distinguishability, continuous functions preserve it.
\end{rem}

\section{Distinguishability of experimental relationships}

To conclude we want to make sure that experimental relationships are themselves experimentally distinguishable. To do so it suffices to show that the set of continuous functions between two Hausdorff, second-countable topological spaces can be given a topology that is Hausdorff and second-countable.

\begin{defn} Let $X$ and $Y$ be two topological spaces. Let $C(X,Y)$ denote the set of all continuous functions from $X$ to $Y$. Let $\mathcal{B}_X$ and $\mathcal{B}_Y$ be two bases for $X$ and $Y$ respectively. The \textbf{basis-to-basis topology} $\mathsf{T}(C(X,Y), \mathcal{B}_X, \mathcal{B}_Y)$ on $C(X,Y)$ with respect to the basis $\mathcal{B}_X$ and $\mathcal{B}_Y$ is the topology generated by all sets of the form 
	$$
	V(U_X, U_Y) = \{f\in C(X,Y) : f(U_X)\subset U_Y\}
	$$
where $U_X \in \mathcal{B}_X$ and $U_Y \in \mathcal{B}_Y$.
\end{defn}

\begin{proof}
	The collection $\mathsf{T}(C(X,Y), \mathcal{B}_X, \mathcal{B}_Y)$ is defined to be the topology generated by the sets $V(U_X,U_Y)$, so it contains the empty set and is closed under arbitrary union and finite intersection by definition. To see why these sets contain every continuous function, let $f\in C(X,Y)$. Then for any $U_Y\in \mathcal{B}_Y$, we can find some $U_X\in \mathcal{B}_X$ such that $U_X\subseteq f^{-1}(U_Y)$. Then $f\in V(U_X,U_Y)$.
\end{proof}

\begin{prop}
	Let $X$ and $Y$ be two Hausdorff and second-countable topological spaces. Let $C(X,Y)$ denote the set of all continuous functions from $X$ to $Y$. Let $\mathcal{B}_X$ and $\mathcal{B}_Y$ be two countable bases for $X$ and $Y$ respectively. The basis-to-basis topology $\mathsf{T}(C(X,Y), \mathcal{B}_X, \mathcal{B}_Y)$ on $C(X,Y)$ with respect to the bases $\mathcal{B}_X$ and $\mathcal{B}_Y$ is Hausdorff and second-countable. 
\end{prop}
\begin{proof}
	First we show that $\mathsf{T}(C(X,Y), \mathcal{B}_X, \mathcal{B}_Y)$ is second-countable. We note that the sub-basis $\{V(U_X, U_Y) \, |\,   U_X \in \mathcal{B}_X , U_Y \in \mathcal{B}_Y \}$ is countable since $\mathcal{B}_X$ and $\mathcal{B}_Y$ are countable and so will be the bases which it generates. This means $\mathsf{T}(C(X,Y), \mathcal{B}_X, \mathcal{B}_Y)$ is second-countable.
	
	Next we show that $\mathsf{T}(C(X,Y), \mathcal{B}_X, \mathcal{B}_Y)$ is Hausdorff. Let $f,g:X\to Y$ be two distinct continuous functions. Then for some $x\in X$, we have $f(x)\neq g(x)$. Pick $V_1, V_2$ disjoint open subsets of $Y$ with $f(x)\in V_1$ and $g(x)\in V_2$. We may assume (possibly by shrinking $V_1$ or $V_2$) that both are basis elements for the topology of $Y$. Let $U=f^{-1}(V_1)\cap g^{-1}(V_2)$. Then $U$ is an open neighborhood of $x$. We may assume again that $U$ is a basis element for the topology on $X$ by shrinking it if necessary. Now, let $T_1$ be the (sub-)basis element for the basis-to-basis topology corresponding to $U$ and $V_1$. By construction, $f\in T_1$. Similarly, let $T_2$ be the basis element for the basis-to-basis topology corresponding to $U$ and $V_2$ and containing $g$. Since $V_1$ and $V_2$ are disjoint, so are $T_1$ and $T_2$. $\mathsf{T}(C(X,Y), \mathcal{B}_X, \mathcal{B}_Y)$ is Hausdorff.
\end{proof}

\begin{rem}
	Note that the basis-to-basis topology is not in general equal to the open-open topology. The former may depend on the choice of bases $\mathcal{B}_X$ and $\mathcal{B}_Y$ while the second is uniquely defined by the topologies of $X$ and $Y$.
\end{rem}

As experimental relationships are themselves distinguishable, we can recursively form experimental relationships between experimental relationships leading to functions of arbitrary order while remaining within the definitions provided. The framework is therefore complete.

\begin{table}[h]
	\centering
	\begin{tabular}{p{0.20\textwidth} p{0.7\textwidth}}
		Math/Topology & Science/Physics \\ 
		\hline 
		Hausdorff, second-countable space & Space of experimentally distinguishable elements, whose points are the possible values and whose open sets represent the experimentally attainable levels of precision. \\
		Open set & Verifiable set. We can verify experimentally that an element is within the set.  \\ 
		Closed set & Refutable set. We can verify experimentally that an element is not in the set. \\ 
		Basis & A collection of verifiable sets that can be used to distinguish an element\\
		Continuous \newline function &  An experimental relationship between two sets of experimentally distinguishable elements, which must preserve distinguishability \\
		Homeomorphism &  A perfect equivalence between spaces of experimentally distinguishable elements. \\
	\end{tabular} 
	\caption{Topology-to-physics dictionary. This table sums up the relationships established between mathematical and scientific concepts.}
\end{table}

\section{Conclusion}

What emerges from this work is that the primary application of topology in science is experimental distinguishability: its role is to keep track of what can be distinguished through experimentation. The importance of continuous functions in science stems from requiring that experimental relationships be consistent with experimental distinguishability. Therefore it is not that the deterministic evolution of a physical system happens to be continuous: it must be.

In light of this work, the effectiveness of mathematics in the natural sciences is perhaps not so unreasonable. We hope that the methods and results shown here can provide a more solid foundation to formalize experimental sciences.

\section*{Acknowledgments}
Funding for this work was provided in part by the MCubed program of the University of Michigan.

\bibliographystyle{elsarticle-num}
\bibliography{bibliography}

\begin{thebibliography}{1}
\expandafter\ifx\csname url\endcsname\relax
  \def\url#1{\texttt{#1}}\fi
\expandafter\ifx\csname urlprefix\endcsname\relax\def\urlprefix{URL }\fi
\expandafter\ifx\csname href\endcsname\relax
  \def\href#1#2{#2} \def\path#1{#1}\fi

\bibitem{wigner}
E.~Wigner, The unreasonable effectiveness of mathematics in the natural
  sciences., Communications on Pure and Applied Mathematics 13~(1) (1960)
  1--14.
\newblock \href {http://dx.doi.org/10.1002/cpa.3160130102}
  {\path{doi:10.1002/cpa.3160130102}}.

\bibitem{rashevsky}
N.~Rashevsky, Topology and life: In search of general mathematical principles
  in biology and sociology, The bulletin of mathematical biophysics 16~(4)
  (1954) 317--348.
\newblock \href {http://dx.doi.org/10.1007/BF02484495}
  {\path{doi:10.1007/BF02484495}}.

\bibitem{chichilnisky}
G.~Chichilnisky, Social choice and the topology of spaces of preferences,
  Advances in Mathematics 37~(2) (1980) 165 -- 176.
\newblock \href
  {http://dx.doi.org/http://dx.doi.org/10.1016/0001-8708(80)90032-8}
  {\path{doi:http://dx.doi.org/10.1016/0001-8708(80)90032-8}}.

\bibitem{nakahara}
M.~Nakahara, Geometry, topology and physics, CRC Press, 2003.

\bibitem{basener}
W.~F. Basener, Topology and its applications, John Wiley \& Sons, 2006.

\bibitem{mill}
J.~S. Mill, A system of logic, Parker, 1843.

\bibitem{russell}
B.~Russell, The problems of philosophy, Home University Library, 1912.

\end{thebibliography}

\end{document}